\documentclass{amsart}

\usepackage[utf8]{inputenc}
\usepackage[T1]{fontenc}

\usepackage{amsfonts}

\usepackage{amssymb}
\usepackage{pst-node}

\usepackage{mathtools}

\usepackage{thmtools}
\usepackage{mathrsfs}
\usepackage{tcolorbox}
\usepackage{listings}
\usepackage{hyperref}
\usepackage{color}
\usepackage[usenames,dvipsnames]{xcolor}
\usepackage{float}
\usepackage{tikz, tikz-cd}
\usepackage{enumerate}
\usepackage[nameinlink, capitalize, noabbrev]{cleveref}
\usepackage{natbib}
\usepackage{caption}
\usetikzlibrary{intersections}
\usepackage{pgfplots}
\pgfplotsset{width=10cm,compat=1.9}
\usepgfplotslibrary{external}
\tikzset{mytext/.style={font=\small, text=black}}
\tikzset{main node/.style={circle,fill=lime!30,draw,minimum size=0.5cm,inner sep=0pt},}

\hypersetup{ colorlinks=true,linkcolor=teal,    citecolor=magenta, }

\usepackage{pst-node}
\usepackage{pst-tree}

\def\BALL[#1](#2){\rput[t](#2){}%
        \pscircle[fillstyle=solid,fillcolor=#1!40](#2){5pt}}
        
\psset{treesep=1.3,levelsep=1.5}

\hypersetup{
    colorlinks=true,
    linkcolor=teal,
    citecolor=magenta,
    }

\numberwithin{equation}{section}

\newtheorem{Theorem}{Theorem}
\newtheorem{Conjecture}[Theorem]{Conjecture}
\newtheorem{Corollary}[Theorem]{Corollary}
\newtheorem{Question}[Theorem]{Question}
\newtheorem{proposition}{Proposition}[section]
\newtheorem{lemma}[proposition]{Lemma}
\newtheorem{corollary}[proposition]{Corollary}
\newtheorem{theorem}[proposition]{Theorem}

\theoremstyle{definition}
\newtheorem{remark}[proposition]{Remark}
\newtheorem{definition}[proposition]{Definition}

\numberwithin{equation}{section}

\title[Arboreal Galois representations of rational functions]{Arboreal Galois representations of rational functions: fixed-point proportion and the extension problem}
\author{Jorge Fariña-Asategui}
\address{Jorge Fariña-Asategui: Centre for Mathematical Sciences, Lund University, 223 62 Lund, Sweden -- Department of Mathematics, University of the Basque Country UPV/EHU, 48080 Bilbao, Spain}
\email{jorge.farina\_asategui@math.lu.se}

\begin{document}

\begin{abstract}

We give an explicit description of the arithmetic-geometric extension of iterated Galois groups of rational functions.  This yields a complete solution to the extension problem when either the arithmetic or the geometric iterated Galois group is branch, answering a question of Adams and Hyde.

Furthermore, we obtain a sufficient condition for the arithmetic iterated Galois group of a rational function to have positive fixed-point proportion, which further applies in many instances to the specialization to non strictly post-critical points. In particular, this holds for all unicritical polynomials of odd degree, which greatly generalizes a result of Radi for the polynomial $z^d+1$.

Lastly, we obtain the first family of groups acting on the $d$-adic tree whose fixed-point process becomes eventually $d$ for any $d\ge 2$ with positive probability. What is more, these groups are fractal and branch and thus positive-dimensional; hence they yield the first family of counterexamples to a conjecture of Jones for every $d$-adic tree.
\end{abstract}

\maketitle

\section{Introduction}
\label{section: introduction}

Let $K$ be a field and $f\in K(z)$ a rational function of degree $d \geq 2$. Let $t$ be a transcendental element over $K$ and fix $K(t)^{\mathrm{sep}}$ a separable closure of $K(t)$. We consider the tree of preimages
$$T := \bigsqcup_{n \geq 0} f^{-n}(t).$$
As $t$ is transcendental, the tree $T$ is the \textit{$d$-adic tree}, i.e. the rooted tree whose vertices all have exactly $d$ immediate descendants. The natural Galois action of $\mathrm{Gal}(K(t)^{\mathrm{sep}}/K(t))$ on $T$, yields the \textit{arboreal Galois representation} 
$$\rho: \mathrm{Gal}(K(t)^{\mathrm{sep}}/K(t)) \to \mathrm{Aut}~T,$$
where $\mathrm{Aut}~T$ denotes the group of automorphisms of $T$. The image
$$G_\infty(K,f,t):=\rho(\mathrm{Gal}(K(t)^{\mathrm{sep}}/K(t)))$$
is called the \textit{arithmetic iterated Galois group of $f$}. Note that
$$G_\infty(K,f,t)\cong \mathrm{Gal}(K_\infty(f,t)/K(t)),$$
where $K_\infty(f,t)$ is the splitting field over $K(t)$ of all the iterates of $f$. Let us define $$L_\infty:=K^{\mathrm{sep}}\cap K_\infty(f,t).$$ 
By Galois correspondence, we have a short exact sequence
\begin{align}
\label{align: extension}
1\longrightarrow G_\infty(L_\infty, f,t) \longrightarrow G_\infty(K,f,t) \longrightarrow \mathrm{Gal}(L_\infty/K) \longrightarrow 1.
\end{align}
The group $G_\infty(L_\infty, f,t)$, which coincides with $G_\infty(K^{\mathrm{sep}}, f,t)$, is the \textit{geometric iterated Galois group of~$f$}.

The extension $L_\infty/K$ is quite mysterious in general; see \cite[Section 3.1]{JonesArboreal} for a detailed discussion. The following, has become a standard problem in the field:

\begin{Question}[Extension problem]
\label{Question: extension problem}
What conditions guarantee that the extension $L_\infty/K$ is finite? 
\end{Question}

The extension problem seems particularly challenging for post-critically finite rational functions. Indeed, an answer to \cref{Question: extension problem} is only known for a very restrictive class of post-critically finite rational functions.

A rational function $f\in K(z)$ is \textit{post-critically finite} if $P_f:=\bigcup_{n\ge 1}f^n(C_f)$ is finite, where $C_f$ denotes the set of critical points of $f$. A polynomial $f\in K[z]$ is \textit{unicritical} if $f$ has a unique critical point $c$ in $K$ and \textit{preperiodic} if $c$ is strictly preperiodic. For a post-critically finite quadratic polynomial $f$, the extension problem was completely solved by Pink in \cite{PinkPol}: \cref{Question: extension problem} has a positive answer for $f$ if and only if $f$ is preperiodic. The work of Pink has been recently extended via a similar approach to unicritical polynomials by Adams and Hyde in \cite{OpheliaHyde}. To the best of our knowledge, these are the only known cases where a positive answer to \cref{Question: extension problem} is known for post-critically finite maps. However, there is also further cases when \cref{Question: extension problem} is known to have a negative answer. For instance, recent work of Hamblen and Jones in \cite{HamblenJones} provides sufficient conditions for the extension $L_\infty/K$ to contain the cyclotomic extension $K(\zeta_{d^\infty})$ and thus $L_\infty/K$ to be infinite (assuming that $K$ is such that $K(\zeta_{d^\infty})/K$ is infinite). 

The first goal of the present paper is to find a sufficient condition for \cref{Question: extension problem} to have a positive answer which holds for a larger class of groups. This is obtained via an explicit description of the group extension in \textcolor{teal}{(}\ref{align: extension}\textcolor{teal}{)}, which generalizes a recent construction of Radi in \cite{SantiFPP}; see \cref{definition: the groups GH}. As we shall see, the class of branch groups plays a key role in \cref{Question: extension problem}.

Let us fix a subgroup $G\le \mathrm{Aut}~T$. For each vertex $v\in T$, we define the corresponding \textit{rigid vertex stabilizer} $\mathrm{rist}_G(v)$ as the subgroup consisting of those elements in $G$ which fix every vertex not a descendant of $v$. Furthermore, for each level $n\ge 1$ the \textit{rigid level stabilizer} $\mathrm{Rist}_G(n)$ is the direct product of the rigid vertex stabilizers of vertices at level $n$. We say that $G$ is \textit{level-transitive} if $G$ acts transitively on every level of $T$. Note that we write $\pi_n(G)$ for the action of $G$ on the $n$th level of $T$. A level-transitive subgroup $G$ whose rigid level stabilizers are of finite index in $G$ is called \textit{branch}. Branch groups were introduced by Grigorchuk in \cite{NewHorizonsGrigorchuk}; see also \cite{WeaklyBranch}. 

Adams and Hyde asked in \cite[page 51]{OpheliaHyde} whether the geometric iterated Galois group being branch is a sufficient condition for $L_\infty/K$ to be finite. We answer their question in the positive; see \cref{section: self-similar extensions} for the definition of a group of finite type and any other unexplained term here and elsewhere in the introduction:

\begin{Theorem}
\label{theorem: extension problem rational functions branch}
Let $K$ be a field, $f\in K(z)$ a rational function and  set $G:=G_\infty(K,f,t)$ and $H:=G_\infty(K^{\mathrm{sep}},f,t)$. Then
\begin{enumerate}[\normalfont(i)]
\item $H$ is branch if and only if $G$ is branch;
\item if either $H$ or $G$ is branch, then
$$[L_\infty:K]=|G:H|= |\pi_D(G):\pi_D(H)|<\infty,$$
where $D$ is the depth of $H$ as a group of finite type. In particular $$L_\infty=K_D(f,t)\cap K^\mathrm{sep},$$
where $K_D(f,t)$ is the splitting field of $\{f^n\}_{n=1}^D$ over $K(t)$.
\end{enumerate}
\end{Theorem}

\cref{theorem: extension problem rational functions branch} applies to a large class of groups. In particular, it applies to both the preperiodic quadratic polynomials considered by Pink in \cite{PinkPol} and to the preperiodic unicritical polynomials considered by Adams and Hyde in \cite{OpheliaHyde} by \cite[Proposition~4.14]{OpheliaHyde}. In fact, to the best of our knowledge, \cref{theorem: extension problem rational functions branch} applies to all the cases where \cref{Question: extension problem} is known to have a positive answer, compare \cite{OpheliaHyde, PinkRat, PinkPol, SantiFPP}. 

We remark that the proof of \cref{theorem: extension problem rational functions branch} is completely different from the approaches of Adams and Hyde in \cite{OpheliaHyde} and Pink in \cite{PinkRat, PinkPol}. Adams and Hyde suggested in \cite[page 51]{OpheliaHyde} to prove that the normalizer of a geometric iterated Galois group $H$ in $\mathrm{Aut}~T$ has finite exponent over $H$ in order to establish \cref{theorem: extension problem rational functions branch}; whereas our approach is based on the construction in \cref{definition: the groups GH} instead. This alternative approach has many further applications. Before discussing them, let us pose a conjecture:

\begin{Conjecture}
\label{Conjecture: branch finite}
Let $f\in K[z]$ be a hyperbolic (in the orbifold sense in \cite{JorgeSanti2}; see also \cite{Thurston, Douady}) polynomial such that $L_\infty\ne K$. Then, the following are equivalent:
\begin{enumerate}[\normalfont(i)]
\item both $G_\infty(K,f,t)$ and $G_\infty(K^\mathrm{sep},f,t)$ are branch;
\item the extension $L_\infty/K$ is finite;
\item none of the critical points of $f$ in $K$ is periodic.
\end{enumerate}
\end{Conjecture}

\cref{Conjecture: branch finite} is motivated by \cref{theorem: extension problem rational functions branch} and the results of Adams and Hyde for unicritical polynomials in \cite{OpheliaHyde}. In fact, \cref{Conjecture: branch finite} holds for unicritical polynomials by the results in \cite{OpheliaHyde} and \cref{theorem: extension problem rational functions branch} yields the implication (i)$\implies$(ii) for every rational function. Furthermore, note that if all the critical points of $f\in K[z]$ are wandering, it can be shown that $G_\infty(K^\infty,f,t)$ is branch, which, together with \cref{theorem: extension problem rational functions branch}\textcolor{teal}{(i)}, further suggests the implication (iii)$\implies$(i).

If $\alpha\in K$ is not a strictly post-critical point for $f$, we may consider the so-called \textit{specialization} $G_\infty(K,f,\alpha)$. This specialization $G_\infty(K,f,\alpha)$ is defined analogously to $G_\infty(K,f,t)$ via its faithful action on the tree of preimages $T_\alpha:=\bigsqcup_{n\ge 0}f^{-n}(\alpha)$. Note that $T_\alpha$ is isomorphic to $T$, as $\alpha$ is not strictly post-critical. Moreover, the group $G_\infty(K,f,\alpha)$ can be embedded into $G_\infty(K,f,t)$ as the subgroup corresponding to the decomposition subgroup in $\mathrm{Gal}(K_\infty(f,t)/K(t))$ of a prime $\mathfrak{P}$ in $K_\infty(f,t)$ above $(t-\alpha)$.

A recurrent question since the pioneering work of Odoni in the 1980s \cite{Odoni1,Odoni2} has been to understand what conditions on $f\in K(z)$ and $\alpha\in K$ guarantee that 
\begin{align}
\label{align: odonis problem}
|\mathrm{Aut}~T:G_\infty(K,f,\alpha)|<\infty.
\end{align}
\cref{theorem: extension problem rational functions branch} yields a necessary condition for this to happen:

\begin{Corollary}
\label{Corollary: finite index in Aut T}
Let $f\in K(z)$ a rational function and let $\alpha\in K$ not strictly post-critical. Then if 
$$|\mathrm{Aut}~T:G_\infty(K,f,\alpha)|<\infty,$$
we necessarily have
$$G_\infty(K^{\mathrm{sep}},f,t)=G_\infty(K,f,t)=\mathrm{Aut}~T.$$
\end{Corollary}

A more subtle question generalizing \cref{align: odonis problem} is the so-called \textit{specialization problem}:

\begin{Question}[Specialization problem]
\label{Question: specialization problem}
What conditions on a not strictly post-critical point $\alpha\in K$ guarantee that $$|G_\infty(K,f,t):G_\infty(K,f,\alpha)|<\infty~?$$
\end{Question}

The specialization problem is expected to be very difficult. In fact, it has only been answered in the positive in either very few particular cases, or conditioned on big conjectures like the abc conjecture; compare \cite{Ahmad,benedetto,Bridy2, Bridy,Bridy3, JonesLMS, Manes1, Looper, Odoni1, Stoll}.

The motivation for Odoni in \cite{Odoni1,Odoni2} to study \cref{align: odonis problem}, and more generally \cref{Question: specialization problem}, comes from a prime density problem and its relation to the fixed-point proportion of the specialization $G_\infty(K,f,0)$; see \cite{Bridy,JorgeSanti2,Juul} for further applications to the proportion of periodic points of $f$ over finite fields. 

Given $a_0\in \mathbb{Q}$ and $f\in \mathbb{Q}(z)$ a rational function, we may consider the set of primes
$$P(f,a_0):=\{p\text{ prime}\mid \nu_p(f^n(a_0))\ne 0\text{ for some }n\ge 0\text{ with }f^n(a_0)\ne 0\}.$$
Then, the natural density, denoted $\mathcal{D}(\cdot)$, of the complement of $P(f,a_0)$ is bounded below by the density of primes $p$ such that $f^n(z)\equiv 0~\text{mod}~p$ has no solution. The latter is equivalent to the Frobenius at $p$ acting fixed-point freely on $f^{-n}(0)$. Therefore, the Chebotarev density theorem implies that 
\begin{align}
\label{align: bound fpp}
\mathcal{D}(P(f,a_0))\le \mathrm{FPP}(G_\infty(\mathbb{Q},f,0)),
\end{align}
where the \textit{fixed-point proportion} of a group $G\le \mathrm{Aut}~T$ is defined as
$$ \mathrm{FPP}(G):=\lim_{n\to\infty}\frac{\#\{g\in \pi_n(G)\mid g \text{ fixes some }v\text{ at level }n\}}{|\pi_n(G)|}.$$

In general, it is expected that $\mathcal{D}(P(f,a_0))=0$ for the majority of rational functions $f\in \mathbb{Q}(z)$. A recent major breakthrough by the author and Radi in \cite{JorgeSanti2}, has completely classified $\mathrm{FPP}(G_\infty(K^\mathrm{sep},f,t))$ for every polynomial $f\in K[z]$.  However, computing the fixed-point proportion of the arithmetic group and its specializations seems to be highly dependent on both the extension and the specialization problems \cite{SantiFPP}. Our approach in the present paper does not require us to understand neither the extension nor the specialization problem in order to show that the fixed-point proportion of the arithmetic group and its specializations is positive. Instead, we obtain a sufficient condition for the positivity of their fixed-point proportion based solely on their monodromy action, i.e. their action on the first level of $T$.

We say that a rational function $f\in K(z)$ has \textit{bad monodromy} if there is a non-empty subset of cosets 
\begin{align}
\label{align: bad monodromy intro}
\emptyset\ne M\subseteq Q:=\frac{\pi_1(G_\infty(K,f,t))}{\pi_1(G_\infty(K^{\mathrm{sep}},f,t))}\cong \mathrm{Gal}(L_1/K)
\end{align}
whose elements all fix a (possibly distinct) preimage in $f^{-1}(t)$. Note that 
$$L_1:=K_1(f,t)\cap K^\mathrm{sep},$$
where $K_1(f,t)$ is the splitting field of $f$ over $K(t)$.

We prove that the arithmetic iterated Galois group of a rational function with bad monodromy and many of its specializations have positive fixed-point proportion. This shows that positive fixed-point proportion is more abundant than previously expected:

\begin{Theorem}
\label{Theorem: bad monodromy}
Let $f\in K(z)$ be a rational function with bad monodromy and $M$ and~$Q$ as in \cref{align: bad monodromy intro}. Then
$$\mathrm{FPP}(G_\infty(K,f,t))\ge \frac{\#M}{|Q|}>0.$$
If $\alpha\in K$ is not strictly post-critical and such that
$$\pi_1(G_\infty(K,f,t))=\pi_1(G_\infty(K,f,\alpha)),$$
we further get
$$\mathrm{FPP}(G_\infty(K,f,\alpha))\ge \frac{\#M}{|Q|}>0.$$
\end{Theorem}

A direct application of \cref{Theorem: bad monodromy} is to unicritical polynomials of odd degree. This yields the following corollary, greatly generalizing a recent result of Radi for the polynomial $z^d+1$ in \cite[Theorem E]{SantiFPP}:

\begin{Corollary}
\label{Corollary: Unicritical fpp}
Let $f\in \mathbb{Q}[z]$ be a unicritical polynomial of odd degree~$d\ge 3$. Then
$$\mathrm{FPP}(G_\infty(\mathbb{Q},f,t))\ge \prod_{p|d}\left(\frac{p-2}{p-1}\right)>0.$$
Furthermore, if $\alpha\in \mathbb{Q}$ is not a strictly post-critical point, then
$$\mathrm{FPP}(G_\infty(\mathbb{Q},f,\alpha))\ge \prod_{p|d}\left(\frac{p-2}{p-1}\right)>0.$$
\end{Corollary}

Therefore, if a unicritical polynomial $f$ of odd degree is either preperiodic or post-critically infinite, then $\mathrm{FPP}(G_\infty(\mathbb{Q},f,0))>0$ and the bound in \cref{align: bound fpp} is positive. Therefore, this approach cannot be used to prove that $\mathcal{D}(P(f,a_0))=0$, suggesting that we need new tools to tackle this problem.

The tools developed in this paper do not only have applications to arboreal Galois representations of rational functions, they also have applications in a more abstract group-theoretic setting. Indeed, let us show an application to the number of fixed-points of random elements in positive-dimensional subgroups of $\mathrm{Aut}~T$.

We define a Hausdorff dimension $\mathrm{hdim}_T(\cdot)$ for closed subgroups of $\mathrm{Aut}~T$ via the metric induced by the action on each level of $T$. If $G\le \mathrm{Aut}~T$ is a closed subgroup, then its Hausdorff dimension coincides with its lower-box dimension by \cite[Theorem 2.4]{BarneaShalev}, i.e.
\begin{align*}   \mathrm{hdim}_T(G)=\liminf_{n\to\infty}\frac{\log|\pi_n(G)|}{\log|\pi_n(\mathrm{Aut}~T)|}.
\end{align*}

In \cite{JonesSurvey}, Jones conjectured that for the binary rooted tree $T$, there are no level-transitive subgroups of $\mathrm{Aut}~T$ with both positive Hausdorff dimension and positive fixed-point proportion, as the only known examples at the time with positive fixed-point proportion were groups of affine transformations. Nevertheless, Boston constructed in \cite{BostonCounterexample} a complicated counterexample to the conjecture of Jones acting on the $2$-adic tree, while Radi further provided simple counterexamples in \cite{SantiFPP} acting on the $d$-adic tree for~$d\not\equiv 2$ mod 4. 

Remarkably, \cref{Corollary: Unicritical fpp} yields the first examples of level-transitive subgroups of $\mathrm{Aut}~T$ with positive fixed-point proportion, which are topologically finitely generated but are not groups of affine transformations. Indeed, if $f\in \mathbb{Q}[z]$ is a post-critically finite unicritical polynomial of odd degree, the group $G_\infty(\mathbb{Q},f,t)$ is topologically finitely generated, has both positive Hausdorff dimension and positive fixed-point proportion, is (weakly) branch and neither is linear nor satisfies any group law; compare \cite{Abert, AbertLinear, OpheliaHyde, JorgeArboreal, QuestionAbertVirag}. Furthermore, in \cref{section: self-similar extensions}, we show that the construction in \cref{definition: the groups GH} applied to the groups $G_\mathcal{S}$, constructed by the author in \cite{RestrictedSpectra}, yields the first counterexamples to the conjecture of Jones in the $d$-adic tree for every $d\ge 2$. In fact, we prove a much stronger result on their fixed-point processes. 

The \textit{fixed-point process} of a group $G\le \mathrm{Aut}~T$, introduced by Jones in \cite{JonesComp}, is the real stochastic process $\{X_n\}_{n\ge 1}$, where the random variables $X_n:(G,\mu_G)\to \mathbb{N}\cup\{ 0\}$ are given by
$$X_n(g):=\# \{\text{vertices at level }n\text{ fixed by }g\}.$$

Under certain circumstances, see for instance \cite[Theorem 5.7]{Bridy}, the fixed-point process is a discrete non-negative martingale and thus it becomes eventually constant almost surely. To the best of our knowledge, in all the known examples, we have $\lim_{n\to \infty}X_n= r$ with $r\in \{0,1,2\}$ almost surely. Furthermore, $r=2$ only appears with positive probability in the geometric iterated Galois groups of Chebyshev polynomials of even degree and the construction of Boston in \cite{BostonCounterexample}. A natural question is the following:

\begin{Question}
\label{question: tucker}
What values $r\ge 0$ can the fixed-point process of a level-transitive group $G\le \mathrm{Aut}~T$ converge to with positive probability?
\end{Question}

We apply the construction in \cref{definition: the groups GH} to the groups $G_\mathcal{S}$ in \cite{RestrictedSpectra} to completely settle \cref{question: tucker}:

\begin{Theorem}
\label{Theorem: arbitrary fixed point process}
Let $d\ge 2$ and $T$ the $d$-adic tree. There exists a fractal branch closed subgroup $G_H\le \mathrm{Aut}~T$ such that
\begin{enumerate}[\normalfont(i)]
\item $\mathrm{hdim}_{T}(G_H)=1/d>0$;
\item $\mathrm{FPP}(G_H)\ge (d-1)!/d^d>0$;
\item there is $S\subseteq G$ with $\mu_G(S)>0$ such that $X_n(g)=d$ for all $n\ge 1$ and $g\in S$.
\end{enumerate}
\end{Theorem}

 \cref{Theorem: arbitrary fixed point process} shows that the family of level-transitive groups with positive fixed-point proportion is far richer than anticipated.

\subsection*{\textit{\textmd{Organization}}} 

\cref{section: self-similar extensions} is devoted to introducing our main construction in \cref{definition: the groups GH} and proving its main properties. This is developed in an abstract setting of self-similar groups and their extensions. We conclude the section with a proof of \cref{Theorem: arbitrary fixed point process}. In \cref{section: arboreal}, we show that the extension problem can be described via the construction introduced in \cref{section: self-similar extensions}. We apply the results in \cref{section: self-similar extensions} to prove \textcolor{teal}{Theorems} \ref{theorem: extension problem rational functions branch} and \ref{Theorem: bad monodromy} and \cref{Corollary: finite index in Aut T}. Lastly, we show applications of our results to unicritical polynomials proving \cref{Corollary: Unicritical fpp}.

\subsection*{\textit{\textmd{Notation}}} We write $\#S$ and $|G|$ for the cardinality of a set $S$ and a group $G$ respectively. Exponential notation will be used for group actions on $T$ and on its space of ends $\partial T$, where we always consider right actions.

\subsection*{Acknowledgements} I would like to thank  Rafe Jones, Santiago Radi and Thomas J. Tucker for very helpful discussions on the fixed-point proportion of arboreal Galois representations. In particular, I would like to thank Thomas J. Tucker for suggesting \cref{question: tucker} during one of these discussions. Lastly, I would like to thank my advisors Gustavo A. Fernández-Alcober and Anitha Thillaisundaram for their support.

\section{Self-similar extensions}
\label{section: self-similar extensions}

In this section, we focus on the group-theoretic side. We define self-similar groups and introduce a construction of intermediate self-similar groups generalizing a construction of Radi in \cite{SantiFPP}. Then, we introduce the fixed-point proportion and the fixed-point process of a group and define what is for a self-similar group to have bad monodromy over another self-similar group. We prove that bad monodromy yields intermediate self-similar groups with positive fixed-point proportion. Furthermore, we show that if the groups are further assumed to be branch, these intermediate self-similar groups are finite extensions of the smaller self-similar group. Lastly, we prove \cref{Theorem: arbitrary fixed point process}.

\subsection{Self-similar groups}

The \textit{$d$-adic tree} $T$ is the infinite tree with root $\emptyset$, where every vertex has exactly $d$ immediate descendants. The set of vertices at distance exactly $n \geq 1$ from the root form the \textit{$n$th level}~$\mathcal{L}_n$. The vertices at distance at most~$n$ from the root form the \textit{$n$th truncated tree}~$T^n$. We denote by $\partial T$ the set of infinite rooted paths in $T$, which we call the \textit{space of ends} of $T$.

We denote the group of automorphisms of $T$ by $\mathrm{Aut}~T$. We write $\mathrm{st}(v)$ and $\mathrm{St}(n)$ for the stabilizer of a vertex $v\in T$ and the pointwise stabilizer of a level $n\ge 1$ respectively.

For $1\le n \le \infty$, a vertex $v\in T$ and $g \in \mathrm{Aut}~T$, the unique automorphism $g|_v^n \in \mathrm{Aut}~T^n$ such that for all $w \in T^n$ we have
$$(vw)^g=v^gw^{g|_v^n},$$ 
is called the \textit{section of $g$ at $v$ of depth $n$}. For $n=1$ we call $g|_v^1$ the \textit{label of $g$ at $v$} and for $n=\infty$ we shall simply write $g|_v$. For every $1\le n \le \infty$, we note that
$$(gh)|_v^n=(g|_v^n) (h|_{v^g}^n)\quad\text{and}\quad (g^{-1})|_{v^g}^n=(g|_v^n)^{-1}.$$

For every $n\ge 1$, we define the map $\pi_n:\mathrm{Aut}~T\to \mathrm{Aut}~T^n$ via
$$\pi_n(g):=g|_\emptyset^n$$
and for every $v\in T$, we define the projection map $\varphi_v: \mathrm{st}(v) \rightarrow \mathrm{Aut}~T$ via $$g \mapsto g|_v.$$

Let us fix a subgroup $G \le \mathrm{Aut}~T$. We define vertex stabilizers and level stabilizers as $\mathrm{st}_G(v) := \mathrm{st}(v) \cap G$ and $\mathrm{St}_G(n) := \mathrm{St}(n) \cap G$ for $v \in T$ and $n \ge 1$ respectively. We say that a group $G \le \mathrm{Aut}~T$ is \textit{level-transitive} if the action of $G$ on each level~$\mathcal{L}_n$ is transitive. A group $G \le \mathrm{Aut}~T$ is \textit{self-similar} if for all $v \in T$ and $g \in G$, we have $g|_v \in G$. Self-similar groups were introduced by Nekrashevych; see \cite{SelfSimilar}. We say that a group $G \leq \mathrm{Aut}~T$ is \textit{fractal} if $G$ is self-similar, level-transitive and 
$$G_v:=\varphi_v(\mathrm{st}_G(v)) = G$$
for all $v \in T$.

\subsection{Intermediate self-similar groups}

We now introduce the main construction in this paper:

\begin{definition}[The groups $G_H$]
\label{definition: the groups GH}
Let $1\ne H\trianglelefteq G\le \mathrm{Aut}~T$ be self-similar subgroups. We define
$$G_H:=\{g\in G\mid (g|_v)(g|_{w})^{-1}\in H\text{ for any }v,w\in T\}.$$
\end{definition}

\begin{remark}
Note that this construction generalizes the groups $G_{\mathcal{Q}}^{\mathcal{P}}$ introduced by Radi in \cite{SantiFPP}. In fact, let $H\le G\le \mathrm{Aut}~T$ be self-similar and set any $1\le n\le \infty$. Then, if $\pi_n(H)\trianglelefteq \pi_n(G)$, we may define
$$G_H^n:=\{g\in \mathrm{Aut}~T\mid g|_v^n\in \pi_n(G)\text{ and }(g|_v^n)(g|_{w}^n)^{-1}\in \pi_n(H)\text{ for any }v,w\in T\}.$$
Therefore, Radi's group $G_{\mathcal{Q}}^{\mathcal{P}}$ corresponds to the $n=1$ case and the group $G_H$ in \cref{definition: the groups GH} corresponds to the $n=\infty$ case. Since the $n=\infty$ case is the natural one arising in arboreal Galois representations of rational functions, as we shall see in \cref{section: arboreal}, we only consider the $n=\infty$ case in the remainder of the paper, for the sake of simplicity. The interested reader might note that most of our results in this section will hold for $G_H^n$ and an arbitrary $1\le n\le \infty$ too. 
\end{remark}

The next proposition yields the basic properties of $G_H$:

\begin{proposition}
\label{proposition: definition GH}
Let $1\ne H\trianglelefteq G\le \mathrm{Aut}~T$ be self-similar subgroups. Then:
\begin{enumerate}[\normalfont(i)]
    \item $G_H$ is a self-similar subgroup such that 
    $$H\trianglelefteq G_H\trianglelefteq G;$$
    \item if $H$ is fractal, then $G_H$ is fractal too.
\end{enumerate}
\end{proposition}
\begin{proof}
Let us prove first that $G_H$ is actually a group, i.e. if $g,h\in G_H$, we only need to check that $gh^{-1}\in G_H$. This follows as in \cite{SantiFPP} from
\begin{align*}
(gh^{-1}|_v)(gh^{-1}|_w)^{-1}&=(g|_v)(h|_{v^{gh^{-1}}})^{-1}(h|_{w^{gh^{-1}}})(g|_{w})^{-1}\\
&\in H^{(g|_v)^{-1}}(g|_v)(g|_w)^{-1}=H
\end{align*}
since $H\trianglelefteq G$ and both $G$ and $H$ are self-similar. An analogous computation shows that $G_H\trianglelefteq G$.

To prove self-similarity of $G_H$, it is enough to note that
$$(g|_u)|_v=g|_{uv}\in G$$
for any $u,v\in T$ as $G$ is self-similar, and
$$((g|_u)|_v)((g|_u)|_w)^{-1}=(g|_{uv})(g|_{uw})^{-1}\in H$$
for any $v,w\in T$. Thus $g|_u\in G_H$ for any $u\in T$. Now $G_H\ge H$ as $H$ is self-similar and hence $(h|_v)(h|_w)^{-1}\in H$ for any $h\in H$. Therefore (i) is proved.

Now, we prove (ii). If $H$ is fractal, we have $H_v=H$. Furthermore $H$ (and thus~$G_H$) is level-transitive. Then, for any $g\in G_H$ and any $v\in T$, there exists $h\in H$ such that $hg\in \mathrm{st}_{G_H}(v)$. As $$(hg)|_v(hg)^{-1} \in H\implies (hg)|_v\in Hg$$
and $H$ is fractal, we further obtain $(G_H)_v=G_H$, i.e. $G_H$ is fractal.
\end{proof}

\subsection{Fixed-point proportion}

For a subgroup $G\le \mathrm{Aut}~T$, we define its \textit{fixed-point proportion} $\mathrm{FPP}(G)$ as
$$\mathrm{FPP}(G):=\lim_{n\to\infty}\frac{\#\{g\in \pi_n(G)\mid g\text{ fixes a vertex in }\mathcal{L}_n\}}{|\pi_n(G)|}.$$
The above limit is well defined by monotone convergence. 

The group $\mathrm{Aut}~T$ is a profinite group with respect to the so-called \textit{congruence topology}, i.e. the topology induced by the level-stabilizer filtration. Every closed subgroup $G\le \mathrm{Aut}~T$ is itself profinite and thus it admits a unique Haar probability measure~$\mu_G$.

As the fixed-point proportion only depends on the congruence quotients $\pi_n(G)$, we have
$$\mathrm{FPP}(G)=\mathrm{FPP}(\overline{G}),$$
where $\overline{G}$ is the closure (with respect to the congruence topology) of $G$ in $\mathrm{Aut}~T$. For a closed subgroup $G\le \mathrm{Aut}~T$, we can describe the fixed-point proportion of $G$ in terms of the Haar probability measure $\mu_G$ as
\begin{align*}
\mathrm{FPP}(G)&=\mu_G(\{g\in G \mid g \text{ fixes an end in }\partial T\}).
\end{align*}

The \textit{fixed-point process} of $G$ is the real stochastic process $\{X_n\}_{n\ge 1}$, where the random variables $X_n:(G,\mu_G)\to \mathbb{N}\cup\{ 0\}$ are given by
$$X_n(g):=\# \{\text{vertices at level }n\text{ fixed by }g\}.$$

A real stochastic process $\{Y_n\}_{n\ge 1}$ defined over a probability space $(X,\mu)$ is a \textit{martingale} if for all $n\ge 1$ we have
\begin{enumerate}[\normalfont(i)]
\item $\mathbb{E}(Y_n)<\infty$;
\item $\mathbb{E}(Y_{n+1}\mid Y_1=t_1,\dotsc, Y_n=t_n)=t_{n}$ for every $t_1,\dotsc, t_n\in \mathbb{R}$ such that $\mu(Y_1=t_1,\dotsc, Y_n=t_n)>0$.
\end{enumerate}

Relevance of martingales comes from a classical convergence result:

\begin{theorem}[Martingale convergence]
Let $\{Y_n\}_{n\ge 1}$ be a non-negative martingale over a probability space $(X,\mu)$ with $\mathrm{E}(Y_1)<\infty$. Then
$$\lim_{n\to\infty} Y_n(x)$$
exists for $x\in X$ almost surely and with finite value.
\end{theorem}

Jones noted in \cite{JonesComp} that the fixed-point process $\{X_n\}_{n\ge 1}$ of a self-similar group~$G$ is usually a martingale; see \cref{lemma: martingale} below. In this case $\{X_n\}_{n\ge 1}$ is eventually constant for almost all $g \in G$, as $X_n(g)$ is a non-negative integer for every $n\ge 1$ and $g\in G$. Let $\mathcal{M}$ be the class of groups acting on a $d$-adic tree whose fixed-point process is a martingale. The class $\mathcal{M}$ was completely classified in \cite{Bridy}:
\begin{lemma}[{see \cite[Theorem 5.7]{Bridy}}]
\label{lemma: martingale}
A group $G\le \mathrm{Aut}~T$ is in $\mathcal{M}$ if and only if for every $n\ge 1$, the subgroup $\mathrm{St}_G(n)$ acts level-transitively on every subtree rooted at level $n$ of $T$.
\end{lemma}

This martingale strategy has been recently combined with a new ergodic theory for self-similar groups developed by the author in \cite{Cyclicity} to obtain a new approach to prove that $\lim_{n\to \infty} X_n=0$ almost surely in many instances; see \cite{JorgeSanti2,JorgeSantiFPP}. In this paper, we study the possible values to which $X_n$ can converge to with positive probability for some $G\le \mathrm{Aut}~T$. In other words, we study the set
$$\mathrm{Spec}(X_n):=\{r\ge 0\mid \lim_{n\to\infty} X_n=r\text{ with positive probability for some }G\le \mathrm{Aut}~T\}.$$

Previously, the only known result on $\mathrm{Spec}(X_n)\subseteq \mathbb{N}\cup \{0\}$ was that it contains the first values
$$\{0,1,2\}\subseteq \mathrm{Spec}(X_n).$$
At the end of the section we shall prove that
$$\mathrm{Spec}(X_n)= \mathbb{N}\cup \{0\},$$ 
completely determining $\mathrm{Spec}(X_n)$. In fact, the groups realizing the values $r\ge 2$ in $\mathrm{Spec}(X_n)$ are all in $\mathcal{M}$, even if this fact is not used in the proof.

\subsection{Bad monodromy}

We now introduce one of the key properties in this paper:

\begin{definition}[Bad monodromy]
\label{definition: bad monodromy}
Let $1\ne H\trianglelefteq G\le \mathrm{Aut}~T$. Their action on $\mathcal{L}_1$ induces a short exact sequence
$$1\longrightarrow \pi_1(H) \longrightarrow \pi_1(G) \longrightarrow Q \longrightarrow 1.$$
We say that $G$ has \textit{complete monodromy over $H$} if $\pi_1(G_H)=\pi_1(G)$. Furthermore, we say that $G$ has \textit{bad monodromy over $H$}, if $G$ has complete monodromy over $H$ and there is a non-empty subset of cosets $\emptyset\ne M\subseteq Q$, where every element in $M$ has a (possibly distinct) fixed-point in $\mathcal{L}_1$.
\end{definition}

Assume in what follows that $1\ne H\trianglelefteq G\le \mathrm{Aut}~T$ and $G$ has complete monodromy over $H$. Let $N$ be the kernel of the composition $\overline{G}_H\to \pi_1(G_H)\to Q$, i.e. we have a short exact sequence
$$1\longrightarrow N \longrightarrow \overline{G}_H \longrightarrow Q \longrightarrow 1.$$
Then, we have $\overline{G}_H=\bigsqcup_{q\in Q}Nq$. By finite additivity and translation invariance of the Haar probability measure $\mu_G$, we get
$$\mathrm{FPP}(\overline{G}_H)=\frac{1}{|Q|}\sum_{q\in Q}\mathrm{FPP}(Nq).$$
Therefore, if there exists at least one $q\in Q$ such that
$$\mathrm{FPP}(qN)>0,$$
then
$$\mathrm{FPP}(G_H)=\mathrm{FPP}(\overline{G}_H)\ge \frac{1}{|Q|}\mathrm{FPP}(Nq)>0,$$
independently of $\mathrm{FPP}(H)$. In other words, if $G$ has complete monodromy over~$H$, the monodromy action alone can make the fixed-point proportion of $G_H$ to be positive. In fact, bad monodromy is a sufficient condition for the fixed-point proportion of $G_H$ to be positive: 

\begin{proposition}
\label{proposition: bad monodromy implies fpp>0}
Let $H\trianglelefteq G\le \mathrm{Aut}~T$ be closed fractal subgroups, such that $G$ has bad monodromy over $H$ with $M$ and $Q$ as in \cref{definition: bad monodromy}. Then
$$\mathrm{FPP}(G_H)\ge \frac{\#M}{|Q|}>0.$$
\end{proposition}
\begin{proof}
As
$$\mathrm{FPP}(G_H)=\mathrm{FPP}(\overline{G}_H)=\frac{1}{|Q|}\sum_{q\in Q}\mathrm{FPP}(Nq),$$
it is enough to show that every element in $Nq$ with $q\in M$ fixes an end in $\partial T$. Let $g\in Nq$. Then
$$(g|_v)(g|_\emptyset)^{-1}\in H$$
for any $v\in T$. Thus, passing to labels we get
$$(g|_v^1)(g|_\emptyset^1)^{-1}\in \pi_1(H).$$
In other words
$$g|_v^1\in \pi_1(H) (g|_\emptyset^1)=\pi_1(H)q\subseteq M.$$
Therefore, every label $g|_v^1$ fixes an immediate descendant of $v$. This implies that $g$ fixes an end in $\partial T$ as wanted. 
\end{proof}

Bad monodromy is not the only way to obtain positive fixed-point proportion. In fact, geometric iterated Galois groups of Chebyshev polynomials of even degree do not have an open normal subgroup which yields bad monodromy. We shall see that the groups $G_H$ put in a similar context the groups of Radi in \cite{SantiFPP} with positive fixed-point proportion and the geometric iterated Galois groups of even-degree Chebyshev polynomials with positive fixed-point proportion. First, we need to discuss branch groups and groups of finite type, which play a critical role in the extension problem in \cref{section: arboreal}.

\subsection{Branchness, regular branchness and groups of finite type}

Let $G\le \mathrm{Aut}~T$ be a level-transitive subgroup. For any vertex $v\in T$, we define the corresponding \textit{rigid vertex stabilizer} $\mathrm{rist}_G(v)$ as the subgroup consisting of those elements of $G$ which only move vertices below $v$. The \textit{rigid level stabilizer} $\mathrm{Rist}_G(n)$ is the direct product of all the rigid vertex stabilizers of the vertices at level $n\ge 1$. We say that $G$ is \textit{branch} if $\mathrm{Rist}_G(n)$ is of finite index in $G$ for every $n\ge 1$.

We say that $G$ is \textit{regular branch over a finite-index subgroup $K$} if $G$ is self-similar, level-transitive and $\mathrm{rist}_K(v)_v\ge K$ for every $v\in T$. For every $D\ge 1$ and $\mathcal{P}\le \mathrm{Aut}~T^D$, we define the corresponding \textit{group of finite type}~$G_\mathcal{P}$ as 
\begin{equation*}
G_\mathcal{P} := \{g \in \mathrm{Aut}~T\mid  g|_v^D \in \mathcal{P} \text{ for all $v \in T$}\}.
\end{equation*}

Well-known results of Grigorchuk and \v{S}uni\'{c} show that being of finite type and regular branch are equivalent for closed level-transitive subgroups:

\begin{theorem}[{{see \cite[Theorem 3]{SunicHausdorff} and \cite[Proposition 7.5]{GrigorchukFinite}}}]
    \label{theorem: equivalence finite type}
    Let $G\le \mathrm{Aut}~T$ be a closed level-transitive group. Then, the following are equivalent:
    \begin{enumerate}[\normalfont(i)]
        \item $G$ is of finite type of depth $D$;
        \item $G$ is regular branch over $\mathrm{St}_G(D-1)$.
    \end{enumerate}
\end{theorem}

Moreover, in \cite{RestrictedSpectra}, the author proved a result that implies that the concepts of branchness and finite type are equivalent for fractal closed groups:

\begin{theorem}[{see \cite[Theorem 3.6]{RestrictedSpectra}}]
\label{theorem: branch and finite type}
Let $G\le \mathrm{Aut}~T$ be a fractal closed subgroup. Then $G$ is branch if and only if $G$ is of finite type.
\end{theorem}

The base group $H$ being fractal and branch imposes strong restrictions to the intermediate group $G_H$, independently of our choice of the fractal overgroup $G$:

\begin{proposition}
\label{proposition: branch extensions}
Let $H\trianglelefteq G\le \mathrm{Aut}~T$ be closed fractal subgroups. If $H$ is branch, then $G_H$ is branch and 
$$|G_H:H|= |\pi_D(G_H):\pi_D(H)|<\infty,$$
where $D$ is the depth of $H$ as a group of finite type.
\end{proposition}
\begin{proof}
Let us assume that $H$ is branch. By \cref{theorem: branch and finite type} the group $H$ is branch if and only if $H$ is of finite type of some depth $D\ge 1$. We prove that 
$$\mathrm{St}_{G_H}(D)\le H,$$
which implies that $G_H$ is of finite type of depth $D+1$ by \cref{theorem: equivalence finite type}, and furthermore 
$$|G_H:H|= |\pi_D(G_H):\pi_D(H)|<\infty.$$
Assume that $g\in \mathrm{St}_{G_H}(D)$. Then 
$$g|_{v}^D=g|_v^D (g|_{\emptyset}^D)^{-1}\in \pi_D(H).$$
As $H$ is of finite type of depth $D$, the inclusion $g|_{v}^D\in \pi_D(H)$ for every $v\in T$ implies that $g\in H$. Therefore
\begin{align*}
\mathrm{St}_{G_H}(D)&\le H.\qedhere
\end{align*}
\end{proof}

\subsection{Proof of \cref{Theorem: arbitrary fixed point process}}

Let us introduce the groups $G_\mathcal{S}$ defined by the author in \cite{RestrictedSpectra}.

We choose inductively a sequence of subgroups $\mathcal{S}:=\{S_n\}_{n\ge 0}$ such that for every $s\in S_n$ we have $s|_v^1=1$ for every vertex $v$ not at level $n$. Note that we may regard $S_n$ as a subgroup of $\mathrm{Sym}(d)\times \overset{d^n}{\ldots}\times \mathrm{Sym}(d)$. We write $D_{d^n}(H)$ for the diagonal embedding of a subgroup  $H\le \mathrm{Sym}(d)$ (or an element $g\in \mathrm{Sym}(d)$) into $\mathrm{Sym}(d)\times \overset{d^n}{\ldots}\times \mathrm{Sym}(d)$.

We need a second sequence of subgroups $\{A_n\}_{n\ge 0}$ in order to choose $\mathcal{S}$. Let us set $A_0:=S_0$. For every $n\ge 1$, we choose the subgroup $S_n$ such that $S_n$ is normalized by the subgroup $A_{n-1}$ and we set $A_n:=S_n\rtimes A_{n-1}$.

Then, we define the closed subgroup $G_\mathcal{S}\le\mathrm{Aut}~T$ as
$$G_\mathcal{S}:=\overline{\langle S_n\mid n\ge 0\rangle}=\varprojlim A_n.$$
Note that we may write any $g\in G_\mathcal{S}$ as an infinite product
$$g=\prod_{n\ge 0}g_n,$$
where each $g_n\in S_n$, i.e. $G_\mathcal{S}=\prod_{n\ge 0}S_n$. Furthermore $\mathrm{St}_{G_\mathcal{S}}(n)=\prod_{k\ge n}S_k$ and  $$G_\mathcal{S}=\mathrm{St}_{G_\mathcal{S}}(n)\rtimes A_{n-1}$$
for every $n\ge 1$. Therefore
$$A_n=\pi_n(G_\mathcal{S})\le \pi_{n+1}(G_\mathcal{S})=A_{n+1}\le G_\mathcal{S}$$
for any $n\ge 1$.

The author provided in \cite[Section 4]{RestrictedSpectra} conditions on $\mathcal{S}$ for the group $G_\mathcal{S}$ to be self-similar, fractal, etc. Combining the results in \cite[Section 4]{RestrictedSpectra} with \cref{lemma: martingale} it is easy to see that $G_\mathcal{S}$ is level-transitive if and only if $G_\mathcal{S}\in \mathcal{M}$. 

We now apply the construction $G_H$ with $H:=G_\mathcal{S}$:

\begin{lemma}
\label{lemma: GS for any k}
Let $H:=G_\mathcal{S}$ with $\mathcal{S}:=\{S_n\}_{n\ge 0}$ such that:
\begin{enumerate}[\normalfont(i)]
\item  $S_0$ is cyclic;
\item  $S_1=D_d(S_{0})$ and $S_n=S_{n-1}\times\dotsb \times S_{n-1}$ for any $n\ge 2$.
\end{enumerate}
For each $\tau\in S_{0}\times \dotsb \times S_{0}$, let $g_\tau\in \mathrm{St}(1)\le \mathrm{Aut}~T$ given by
$$g_\tau=\tau \prod_{n\ge 1}D_{d^n}(\tau)$$
and define $G\le \mathrm{Aut}~T$ via
$$G:=\langle H, g_\tau\mid \tau\in S_{0}\times \dotsb \times S_{0}\rangle.$$
Then $G=G_H$ and $H\ge \mathrm{St}_G(2)$.
\end{lemma}
\begin{proof}
First, note that by (ii) and \cite[Proposition 4.1]{RestrictedSpectra}, the group $H$ is self-similar. Let us fix $\tau\in S_{0}\times \dotsb \times S_{0}$. For any $n\ge 1$ and any $v\in \mathcal{L}_n$, we define
\begin{align*}
\sigma_v:=g_\tau|_v^1\in \pi_{1}(H).
\end{align*}
By the definition of $g_\tau$, for any $v\in T$, we get
$$g_\tau|_v=\sigma_v\tau\prod_{n\ge 1}D_{d^n}(\tau)=\sigma_v g_\tau.$$
Therefore
$$(g_\tau|_v)(g_\tau|_w)^{-1}=\sigma_vg_\tau g_\tau^{-1}\sigma_w^{-1}=\sigma_v\sigma_w^{-1}\in  \pi_1(H)\le H.$$

Now, note that
$$\pi_{2}(H)\trianglelefteq \pi_{2}(G)$$
by (ii). Indeed
$$S_1=D_d(S_{0})\trianglelefteq S_{0}\times\dotsb\times S_{0}$$
since $S_{0}$ is abelian by (ii), and thus
$$\pi_{2}(H)=S_1 \rtimes S_0 \trianglelefteq (S_{0}\times\dotsb \times S_{0})\rtimes S_0 =\pi_{2}(G)$$
as $\pi_1(H)=\pi_1(G)=S_0$. Now for any $h\in H$ and any $\tau\in S_{0}\times\dotsb\times S_{0}$, we get
$$(h^{g_\tau})|_v=(g_\tau|_{v^{g_\tau^{-1}}})^{-1}(h|_{v^{g_\tau^{-1}}})(g|_{v^{{g_\tau^{-1}}h}})=g_\tau^{-1}\big(\sigma_{v^{g_\tau^{-1}}}^{-1}(h|_{v^{g_\tau^{-1}}}) \sigma_{v^{g_\tau^{-1} h}}\big) g_\tau$$
and thus
$$(h^{g_\tau})|_v^{2}=\pi_{2}(g_\tau)^{-1} \pi_{2}\big(\sigma_{v^{g_\tau^{-1}}}^{-1}(h|_{v^{g_\tau^{-1}}}) \sigma_{v^{g_\tau^{-1} h}}\big)\pi_{2}(g_\tau)\in \pi_{2}(H)^{\pi_{2}(G)}=\pi_{2}(H).$$
By (ii) and \cite[Theorem 3.4]{RestrictedSpectra}, the group $H$ is of finite type of depth $2$, so $h^{g_\tau}\in H$. Hence $H\trianglelefteq G$ and it follows that $G=G_H$. Lastly, the same argument as in the proof of \cref{proposition: branch extensions} yields
\begin{align*}
H&\ge \mathrm{St}_G(2).\qedhere
\end{align*}
\end{proof}

Let $\mathrm{hdim}_T(\cdot)$ be the Hausdorff dimension defined for closed subgroups of $\mathrm{Aut}~T$ in the introduction. Recall that if $G\le \mathrm{Aut}~T$ is a closed subgroup, then
\begin{align*}   \mathrm{hdim}_T(G)=\liminf_{n\to\infty}\frac{\log|\pi_n(G)|}{\log|\pi_n(\mathrm{Aut}~T)|}
\end{align*}
and that $\mathrm{hdim}_T(H)=\mathrm{hdim}_T(G)$ if $H$ is a finite-index subgroup of $G$.

Now, we are in position to prove \cref{Theorem: arbitrary fixed point process}, namely that for any $d\ge 2$ there exists a fractal branch closed group $G_H$ acting on the $d$-adic tree with positive Hausdorff dimension, positive fixed-point proportion and with fixed-point proportion becoming eventually $d$ with positive probability:

\begin{proof}[Proof of \cref{Theorem: arbitrary fixed point process}]

Let $H$ and $G$ be those in \cref{lemma: GS for any k}. Furthermore, the group~$G$ may be assumed to be fractal, branch and with $\mathrm{hdim}_T(G)=1/d>0$. Indeed, the group $H$ is fractal, branch and has $\mathrm{hdim}_T(H)=1/d>0$ by the proof of \cite[Theorem A]{RestrictedSpectra}, and then we may apply \cref{proposition: definition GH}\textcolor{teal}{(ii)} and \cref{proposition: branch extensions} to obtain that $G$ is also fractal, branch and has $\mathrm{hdim}_T(G)=1/d>0$. Let us prove~(iii) in the statement of \cref{Theorem: arbitrary fixed point process}.

Let $\sigma\in S_0$ be a cycle of order $d$ generating $S_0$. We consider
$$\tau^\rho=(\sigma^{1^\rho},\sigma^{2^\rho},\dotsc, \sigma^{d^\rho}),$$
for any $\rho\in \mathrm{Sym}(d)$. Then, for any $h\in \mathrm{St}_H(1)$, we get
\begin{align*}
\pi_2(hg_\tau^\rho)&=\pi_2(h)\pi_2(g_\tau^\rho)=(\sigma^j,\dotsc, \sigma^j)(\sigma^{1^\rho},\sigma^{2^\rho},\dotsc, \sigma^{d^\rho})\\
&=(\sigma^{1^\rho+j},\sigma^{2^\rho+j},\dotsc, \sigma^{d^\rho+j})
\end{align*}
and the element $hg_\tau^\rho$ has trivial label at exactly one vertex $v\in \mathcal{L}_1$. Thus
$$(hg_\tau^\rho)|_v\in Hg_\tau^\rho\cap \mathrm{St}(1)=   \mathrm{St}_H(1)g_\tau^\rho.$$
This implies that for each $h\in \mathrm{St}_H(1)$, there is a unique end $\gamma\in \partial T$ such that:
\begin{enumerate}[\normalfont(i)]
\item the end $\gamma$ is fixed by $hg_\tau^\rho$;
\item $(hg_\tau^\rho)|_v^1=1$ for any $v\in \gamma$;
\item $hg_\tau^\rho$ moves every vertex $w$ not an immediate descendant of a vertex $v\in \gamma$.
\end{enumerate}
Therefore
$$X_n(hg_\tau^\rho)=d$$
for every $n\ge 1$. In particular, by \cref{lemma: GS for any k} we have $\mathrm{St}_H(1)\ge \mathrm{St}_G(2)$ and thus
\begin{align*}
\mathrm{FPP}(G)&\ge \mu_G\Big(\bigsqcup_{\rho\in \mathrm{Sym}(d)} g_\tau^\rho\mathrm{St}_G(2)\Big)=\frac{|\mathrm{Sym}(d)|}{|\pi_{2}(G)|}=\frac{d!}{d \cdot d^d}\\
&=\frac{(d-1)!}{d^d}>0.
\end{align*}
This yields (ii) in the statement of \cref{Theorem: arbitrary fixed point process}, concluding the proof.
\end{proof}

Note that a similar property is satisfied by the geometric iterated Galois group of a Chebyshev polynomial of even degree. Hence, \cref{Theorem: arbitrary fixed point process} yields an alternative way to prove that the fixed-point proportion of even-degree Chebyshev polynomials is positive in the same spirit as \cref{Theorem: bad monodromy}. Note that previous computations of the fixed-point proportion of these geometric iterated Galois groups were carried out on dense subsets instead, i.e. on the corresponding iterated monodromy groups; see \cite{JorgeSanti2, JonesAMS}.

\section{Arboreal Galois representations of rational functions}
\label{section: arboreal}

In this section, we relate the extension problem of iterated Galois groups to the self-similar extensions in \cref{definition: the groups GH}. We then apply the results in \cref{section: self-similar extensions} to arboreal Galois representations and completely solve the extension problem for branch groups, proving \cref{theorem: extension problem rational functions branch}. Moreover, we show an application of the results on bad monodromy in \cref{section: self-similar extensions} to the fixed-point proportion of arithmetic iterated Galois groups and their specializations even when the specialization problem is unknown, which concludes the proof of \cref{Theorem: bad monodromy}. Lastly, we apply our results to the family of unicritical polynomials and establish \cref{Corollary: Unicritical fpp}.

\subsection{Arithmetic and geometric iterated Galois groups}

Let $K$ be a field and $f\in K(z)$ a rational function of degree $d \geq 2$. Let $t$ be a transcendental element over $K$ and fix $K(t)^{\mathrm{sep}}$ a separable closure of $K(t)$.  We write $f^n:=f\circ\overset{n}{\ldots}\circ f$ and $f^{-n}(t):=(f^n)^{-1}(t)$ for any $n\ge 0$. Let us consider the tree of preimages
$$T := \bigsqcup_{n \geq 0} f^{-n}(t),$$
which is $d$-adic as $t$ is transcendental. The Galois action of $\mathrm{Gal}(K(t)^{\mathrm{sep}}/K(t))$ on each $f^{-n}(t)$ induces an action by automorphisms on~$T$, which yields the \textit{arboreal Galois representation} 
$$\rho: \mathrm{Gal}(K(t)^{\mathrm{sep}}/K(t)) \to \mathrm{Aut}~T.$$
The image
$$G_\infty(K,f,t):=\rho(\mathrm{Gal}(K(t)^{\mathrm{sep}}/K(t)))$$
is called the \textit{arithmetic iterated Galois group of $f$}. Let us define 
$$L_\infty:=K^{\mathrm{sep}}\cap K_\infty(f,t),$$
where $K_\infty(f,t)$ is the splitting field of all the iterates of $f$ over $K(t)$. Then, Galois correspondence induces a short exact sequence
$$1\longrightarrow G_\infty(L_\infty, f,t) \longrightarrow G_\infty(K,f,t) \longrightarrow \mathrm{Gal}(L_\infty/K) \longrightarrow 1.$$
Note that we have the equality
$$G_\infty(L_\infty, f,t)=G_\infty(K^{\mathrm{sep}}, f,t).$$
The group $G_\infty(K^{\mathrm{sep}}, f,t)$ is called the \textit{geometric iterated Galois group of $f$}.

\subsection{Self-similarity and fractality}

Both the geometric and the arithmetic iterated Galois groups of a rational function $f\in K(z)$ are level-transitive closed subgroup of $\mathrm{Aut}(T)$, as the Galois actions are transitive; see \cite[Section 1.1]{JonesArboreal}. Furthermore, they are fractal subgroups of $\mathrm{Aut}~T$, as we shall show now following the presentation of Adams and Hyde in \cite[Section 3]{OpheliaHyde}.

Given $t,t'$ transcendental over $K$, a \textit{path} $\lambda$ is a $K^{\mathrm{sep}}$-isomorphism $$\lambda: K(t)^{\mathrm{sep}} \rightarrow K(t')^{\mathrm{sep}}$$
such that $t^\lambda = t'$. Hence, a path $\lambda$ induces an $L_\infty$-isomorphism between $K_\infty(f,t)$ and $K_\infty(f,t')$. 

Let us fix a set of paths $\Lambda = \{\lambda_i\}_{i=1}^d$, where $\lambda_i$ is a path from $t$ to $t_i$ for each $t_i\in f^{-1}(t)$. If $g$ is an element in $G_\infty(K,f,t)$ and $t_i, t_j \in f^{-1}(t)$ are such that $t_i^g = t_j$, then the restriction of $g$ to $K_\infty(f,t_i)$, which we denote $g_i$, is a $L_\infty(t)$-isomorphism from $K_\infty(f,t_i)$ to $K_\infty(f,t_j)$. Therefore
$$g|_{t_i}= \lambda_i \cdot g_i \cdot \lambda_j^{-1} \in G_\infty(K,f,t),$$
i.e. the set of paths $\Lambda$ induces a self-similar embedding of $G_\infty(K,f,t)$ into $\mathrm{Aut}~T$:
\begin{align*}
G_\infty(K,f,t) & \hookrightarrow G_\infty(K,f,t)^d \rtimes G_1(K,f,t) \le \mathrm{Aut}~T \\
g & \mapsto (g|_{t_1}, \dots, g|_{t_d}) \pi_1(g),
\end{align*}
where $\pi_1(g)$ is given by the Galois action of $g$ on $f^{-1}(t)$. Moreover, by the lifting property for normal extensions, see \cite[Theorem 7]{Ophelia}, the map 
\begin{align*}
\mathrm{st}_{G_\infty(K,f,t)}(t_i) & \rightarrow G_\infty(K,f,t) \\
g & \mapsto g|_i.
\end{align*}
is surjective for every $t_i\in f^{-1}(t)$. As this holds for every field $K$, both the geometric and the arithmetic Galois groups are always fractal:

\begin{proposition}[{see \cite[Theorem 7]{Ophelia}}]
Let $K$ be a field and $f\in K(z)$ a rational function. The groups $G_\infty(K,f,t)$ and $G_\infty(K^\mathrm{sep},f,t)$ are fractal as subgroups of $\mathrm{Aut}~T$.
\end{proposition}

Therefore, we have fractal subgroups $1\ne G_\infty(K^\mathrm{sep},f,t)\trianglelefteq G_\infty(K,f,t)\le \mathrm{Aut}~T$ and we may apply the theory developed in \cref{section: self-similar extensions}.

\subsection{Base field extensions}

Now, more generally for any $1\le n\le \infty$, we consider
$$L_n:=K^{\mathrm{sep}}\cap K_n(f,t),$$
where $K_n(f,t)$ is the splitting field of $\{f^k\}_{k\le n}$. Again by Galois correspondence, we get a short exact sequence
$$1\longrightarrow G_\infty(L_n, f,t) \longrightarrow G_\infty(K,f,t) \longrightarrow \mathrm{Gal}(L_n/K) \longrightarrow 1$$
corresponding to a base field extension. Let us write
$$G_{n}^{K}:=G_\infty(K, f,t)_{G_\infty(L_n,f,t)}$$
for the construction in \cref{definition: the groups GH} applied to the groups in the short exact sequence above. In particular $G_\infty^K=G_\infty(K, f,t)_{G_\infty(K^{\mathrm{sep}},f,t)}$.

We now show that the arithmetic iterated Galois group $G_\infty(K,f,t)$ coincides with $G_n^K$ for any $1\le n\le \infty$:

\begin{proposition}
\label{proposition: arithemtic description}
We have  $G_\infty(K,f,t)=G_{n}^{K}$ for any $1\le n\le \infty$.
\end{proposition}
\begin{proof}
As $G_\infty(K,f,t)$ is a group extension of $G_\infty(L_n,f,t)$ by $\mathrm{Gal}(L_n/K)$, it is enough to show that 
$$(g|_v)(g|_w)^{-1}\in G_\infty(L_n, f,t),$$
for any $g\in G_\infty(K, f,t)$ and any $v,w\in T$. Let us write $t_v$ for the (transcendental) preimage of $t$ by an iterate of $f$ corresponding to the vertex $v\in T$ and $g_v$ for the restriction of $g$ to $K_\infty(f,t_v)$. Note that the restrictions of $g_v$ and of $g$ to $L_n$ coincide, and hence the restrictions of $g_v$ and $g_w$ to $L_n$ coincide. Indeed, we have $$L_n\subseteq K_\infty(f,t_v)\subseteq K_{\infty}(f,t)$$
and as $g_v$ is the restriction of $g$ to $K_\infty(f,t_v)$, we must have $g|_{L_n}=g_v|_{L_n}$, where we write $\cdot|_{L_n}$ for the restriction to the subfield $L_n$ (not to be confused with the notation for sections!). Therefore, as for every $u\in T$ the path $\lambda_u:K[t_u]^{\mathrm{sep}}\to K[t]^{\mathrm{sep}}$ is a $K^{\mathrm{sep}}$-isomorphism and 
$$g|_v=\lambda_v^{-1} g_v \lambda_{v^g},$$
we get
\begin{align*}
(g|_v)|_{L_n}&=(\lambda_v^{-1}|_{L_n})(g_v|_{L_n})(\lambda_{v^g}|_{L_n})=g|_{L_n}\\
&=(\lambda_w^{-1}|_{L_n})(g_w|_{L_n})(\lambda_{w^g}|_{L_n})\\
&=(g|_w)|_{L_n}.
\end{align*}
In other words, both $g|_v$ and $g|_w$ restrict to the same automorphism of $L_n/K$. Thus 
$$(g|_v)(g|_w)^{-1}\in G_\infty(L_n, f,t)$$
as wanted.
\end{proof}

\cref{proposition: arithemtic description} allows us to use the results in \cref{section: self-similar extensions} to study the arboreal Galois representations of a rational function.

\subsection{Extension problem}

We now prove an improvement of \cref{proposition: branch extensions}.  This yields a complete solution to the extension problem for branch groups, proving \cref{theorem: extension problem rational functions branch} and answering a question of Adams and Hyde in \cite[page 51]{OpheliaHyde}:

\begin{proof}[Proof of \cref{theorem: extension problem rational functions branch}]
Recall that $H:=G_\infty(K^\mathrm{sep},f,t)$ and $G:=G_\infty(K,f,t)$. If $H$ is branch, then \cref{proposition: branch extensions}, together with \cref{proposition: arithemtic description}, implies that $G$ is branch and (ii) is satisfied. Hence, let $G$ be branch, and thus by \cref{theorem: branch and finite type}, of finite type of some depth $D\ge 1$. It only remains to show that $H$ is also of finite type of depth~$D$. For any $n\ge 1$, let us write 
$$G_n:=G_\infty(L_n,f,t).$$
As $G_n$ is open in $G$, we have $G_n\ge \mathrm{St}_G(n)$ for some $n\ge D$. As $G$ is of finite type of depth $D$, the level stabilizer $\mathrm{St}_G(D-1)$ is branching by \cref{theorem: equivalence finite type}. Thus, there is a vertex $v\in T$, where
$$(G_n)_v\ge \mathrm{St}_G(n)_v\ge \mathrm{St}_G(D-1).$$
Therefore, as $G_n$ is fractal, we must have $G_n\ge \mathrm{St}_G(D-1)$. Now, by Galois correspondence we have a filtration
$$G\ge G_1\ge G_2\ge \dotsb \ge G_n\ge \dotsb\ge \bigcap_{n\ge 1}G_n=H.$$
Hence
$$H=\bigcap_{n\ge 1}G_n\ge \bigcap_{n\ge 1} \mathrm{St}_G(D-1)\ge \mathrm{St}_G(D-1)$$
and 
$$\mathrm{St}_H(D-1)=\mathrm{St}_G(D-1)$$
is branching. Therefore $H$ is of finite type of depth $D$ by \cref{theorem: equivalence finite type}.
\end{proof}

The author proved in \cite{RestrictedSpectra} that fully-dimensional self-similar subgroups of an iterated wreath product cannot be proper. We restate here the relevant version for fully-dimensional subgroups of $\mathrm{Aut}~T$ that we need:

\begin{proposition}[{see \cite[Subsection 3.6]{RestrictedSpectra}}]
\label{proposition: hdim 1 in ss}
Let $G\le \mathrm{Aut}~T$ be a closed self-similar subgroup such that $\mathrm{hdim}_T(G)=1$. Then
$$G=\mathrm{Aut}~T.$$
\end{proposition}

We may use \cref{proposition: hdim 1 in ss} to prove \cref{Corollary: finite index in Aut T}, i.e. that a finite-index specialization forces the geometric iterated Galois group to be the full automorphism group of $T$:

\begin{proof}[Proof of \cref{Corollary: finite index in Aut T}]
Let us assume that $|\mathrm{Aut}~T:G_\infty(K,f,\alpha)|<\infty$. Then, as $G_\infty(K,f,\alpha)$ embeds into $G_\infty(K,f,t)$, we further get $|\mathrm{Aut}~T:G_\infty(K,f,t)|<\infty$. In particular, this implies that $\mathrm{hdim}_T(G_\infty(K,f,t))=1$. Since $G_\infty(K,f,t)$ is self-similar and of Hausdorff dimension 1 in $\mathrm{Aut}~T$, we may apply \cref{proposition: hdim 1 in ss} to obtain that
$$G_\infty(K,f,t)=\mathrm{Aut}~T.$$
Furthermore, by \cref{theorem: extension problem rational functions branch}, the extension $L_\infty/K$ is finite as $\mathrm{Aut}~T$ is branch. Therefore, the self-similar subgroup $G_\infty(K^{\mathrm{sep}},f,t)$ has Hausdorff dimension 1 in $G_\infty(K,f,t)=\mathrm{Aut}~T$. Again by \cref{proposition: hdim 1 in ss}, we get
\begin{align*}
G_\infty(K^{\mathrm{sep}},f,t)&=\mathrm{Aut}~T.\qedhere
\end{align*}
\end{proof}

We note that an alternative proof of \cref{Corollary: finite index in Aut T} could be obtained using the results of the author on the Hausdorff dimension of self-similar groups in \cite{RestrictedSpectra} and their normal subgroups in \cite{QuestionAbertVirag} instead.

\subsection{Fixed-point proportion and bad monodromy}

We say that $f\in K(z)$ has \textit{bad monodromy} if the arithmetic iterated Galois group $G_{\infty}(K,f,t)$ has bad monodromy over the geometric iterated Galois group $G_{\infty}(K^\mathrm{sep},f,t)$ (or equivalently over $G_{\infty}(L_1,f,t)$) in the sense of \cref{definition: bad monodromy}. Combining \textcolor{teal}{Propositions}~\ref{proposition: bad monodromy implies fpp>0} and \ref{proposition: arithemtic description} we get:

\begin{corollary}
\label{corollary: bad monodromy rational functions}
Let $f\in K(z)$ be a rational function with bad monodromy and let $M$ and $Q$ be given by \cref{definition: bad monodromy}. Then
$$\mathrm{FPP}(G_\infty(K,f,t))\ge \frac{\#M}{|Q|}>0.$$
\end{corollary}

Let $f\in K(z)$ and $\alpha\in K$ not strictly post-critical for $f$. Then, we may define a Galois action of $\mathrm{Gal}(K^{\mathrm{sep}}/K)$ on the $d$-adic tree $T$ of preimages of $\alpha$. This yields an arboreal Galois representation whose image in $\mathrm{Aut}~T$ is denoted $G_\infty(K,f,\alpha)$. Note that, by a standard specialization argument, the group $G_\infty(K,f,\alpha)$ embeds into $G_\infty(K,f,t)$ as the decomposition subgroup of a prime $\mathfrak{P}$ in $K_\infty(f,t)$ above the prime $(t-\alpha)$ in $K(t)$. We prove that bad monodromy is enough to get positive fixed-point proportion in the specialization:

\begin{corollary}
\label{corollary: specialization fpp}
Let $f\in K(z)$ be a rational function with bad monodromy and let $M$ and $Q$ be given by \cref{definition: bad monodromy}. Let $\alpha\in K$ not strictly post-critical and such that
$$\pi_1(G_\infty(K,f,t))=\pi_1(G_\infty(K,f,\alpha)).$$
Then
$$\mathrm{FPP}(G_\infty(K,f,\alpha))\ge \frac{\#M}{|Q|}>0.$$
\end{corollary}
\begin{proof}
As $\alpha\in K$ is not strictly post-critical, the group $G_\infty(K,f,\alpha)$ embeds into $G_\infty(K,f,t)$ as the decomposition subgroup of a prime $\mathfrak{P}$ in $K_\infty(f,t)$ above $(t-\alpha)$. The assumption 
$$\pi_1(G_\infty(K,f,t))=\pi_1(G_\infty(K,f,\alpha))$$
implies that
$$G_\infty(K,f,\alpha)=\bigsqcup_{q\in Q}G_\infty(L_1,f,\alpha)q,$$
with all cosets $G_\infty(L_1,f,\alpha)q$ non-empty. Thus, as 
$$G_\infty(L_1,f,\alpha)M\subseteq G_\infty(K,f,t)M$$
and every element in $G_\infty(K,f,t)M$ fixes an end in $\partial T$, we get
\begin{align*}
\mathrm{FPP}(G_\infty(K,f,\alpha))&\ge \frac{\#M}{|Q|}>0.\qedhere
\end{align*}
\end{proof}

\textcolor{teal}{Corollaries} \ref{corollary: bad monodromy rational functions} and \ref{corollary: specialization fpp} conclude the proof of \cref{Theorem: bad monodromy}. To conclude the paper, we give specific examples of rational polynomials with bad monodromy among unicritical polynomials.

\subsection{An example: unicritical polynomials}

Let $f\in K[z]$ be a polynomial. We say that $f$ is \textit{unicritical} if $f$ has a unique critical point in $K$. Note that a unicritical polynomial $f\in K[z]$ is linearly conjugate to a polynomial of the form $z^d+c$ with $c\in K$. 

We show that the result on the positivity of the fixed-point proportion of arithmetic iterated Galois groups obtained by Radi in \cite{SantiFPP} for the polynomial $x^d+1$ holds more generally for any unicritical polynomial $x^d+c$.

In what follows, we set $f(z):=z^d+c$ with $c\in K$. It is well known that
$$K_1(f,t)=K(\zeta_d,\sqrt[d]{t-c}),$$
where $\zeta_d$ is a $d$th primitive root of unity. Then
$$\pi_1(G_\infty(K,f,t))\cong\mathrm{Gal}(K_1(f,t)/K)= \mathrm{Gal}(K(\zeta_d,\sqrt[d]{t-c})/K)$$
and 
$$\pi_1(G_\infty(L_1, f, t))\cong \mathrm{Gal}(K(\zeta_d,\sqrt[d]{t-c})/K(\zeta_d)).$$
Let us assume now that $K:=\mathbb{Q}$ (or, more generally $K$ is a number field such that $\mathrm{Gal}(K(\zeta_d)/K)\cong (\mathbb{Z}/d\mathbb{Z})^\times$). As
$$\frac{\pi_1(G_\infty(\mathbb{Q},f,t))}{\pi_1(G_\infty(L_1, f, t))}\cong \mathrm{Gal}(\mathbb{Q}(\zeta_d)/\mathbb{Q})\cong (\mathbb{Z}/d\mathbb{Z})^\times,$$
we get a short exact sequence
$$1\longrightarrow G_\infty(L_1, f,t) \longrightarrow G_\infty(\mathbb{Q},f,t) \longrightarrow (\mathbb{Z}/d\mathbb{Z})^\times \longrightarrow 1.$$
In fact, the above yields an affine monodromy action
$$\pi_1(G_\infty(\mathbb{Q},f,t))\cong \mathrm{Aff}(d),$$
where
$$\mathrm{Aff}(d):=\{z\mapsto az+b\mid a\in (\mathbb{Z}/d\mathbb{Z})^\times\text{ and }b\in \mathbb{Z}/d\mathbb{Z}\}.$$
Under this identification, we have 
$$\pi_1(G_\infty(L_1, f, t))\cong \{z\mapsto z+b\mid b\in \mathbb{Z}/d\mathbb{Z}\}\trianglelefteq \mathrm{Aff}(d).$$

The group $\pi_1(G_\infty(\mathbb{Q},f,t))$ has complete monodromy over $\pi_1(G_\infty(L_1, f,t))$. Let $M\subset (\mathbb{Z}/d\mathbb{Z})^\times$ denote the set of bad cosets of $\pi_1(G_\infty(\mathbb{Q},f,t))$ over $\pi_1(G_\infty(L_1, f,t))$. Radi showed in \cite[Proposition 5.2]{SantiFPP} that 
$$\# M=d\prod_{p|d}\left(1-\frac{2}{p}\right)$$
and
$$\frac{\#M}{|(\mathbb{Z}/d\mathbb{Z})^\times|}=\frac{\# M}{\varphi(d)}=\prod_{p|d}\left(\frac{p-2}{p-1}\right)$$
where $\varphi(\cdot )$ denotes the Euler totient function. Therefore, if $d\ge 3$ is odd, \cref{Theorem: bad monodromy} yields
$$\mathrm{FPP}(G_\infty(\mathbb{Q},f,t))\ge \prod_{p|d}\left(\frac{p-2}{p-1}\right)>0$$
generalizing \cite[Theorem E]{SantiFPP} and establishing the first part of \cref{Corollary: Unicritical fpp}.

Now, let us consider $\alpha \in \mathbb{Q}$ not strictly post-critical for $f(z)=z^d+c$. As $\alpha\ne c=f(0)$, we have
$$f(z)-\alpha=\prod_{i=1}^{d} (z-\zeta_d^i \sqrt[d]{\alpha-c}).$$
Then, the monodromy group of the specialization $\pi_1(G_\infty(\mathbb{Q},f,\alpha))$ is isomorphic to the affine group $\mathrm{Aff}(d)$, i.e.
$$\pi_1(G_\infty(\mathbb{Q},f,\alpha))=\pi_1(G_\infty(\mathbb{Q},f,t)).$$
Hence
$$\mathrm{FPP}(G_\infty(\mathbb{Q},f,\alpha))\ge \prod_{p|d}\left(\frac{p-2}{p-1}\right)>0$$
by \cref{Theorem: bad monodromy}, concluding the proof of \cref{Corollary: Unicritical fpp}.



\bibliographystyle{unsrt}

\end{document}